\documentclass[11pt]{amsart}

\usepackage{amsfonts, amscd,color,amsmath,amssymb,amsthm}
\usepackage{color, graphicx}

\usepackage{tikz}

 \DeclareMathOperator{\cl}{cl}
 \DeclareMathOperator{\bd}{bd}

\newcommand{\R}{\mathbb{R}}
     \newcommand{\N}{\mathbb{N}}

\newcommand{\Q}{\mathbb{Q}}

\newcommand{\Pa}{\mathcal{P}}

\newcommand{\K}{\mathcal K}
\newcommand{\C}{\mathbb{C}}

\newtheorem{theorem}{Theorem}[section]

\newtheorem{corollary}[theorem]{Corollary}
\newtheorem{proposition}[theorem]{Proposition}

\newtheorem{fact}[theorem]{Fact}

\theoremstyle{definition}

\theoremstyle{remark}

\newtheorem{remark}[theorem]{Remark}

\newtheorem{question}[theorem]{Question}

\numberwithin{equation}{section}

\begin{document}

\title{On hyperspaces of knots and planar simple closed curves}

\author{Pawe{\l}  Krupski}

\email{pawel.krupski@pwr.edu.pl}
\address{Faculty of Pure and Applied Mathematics, Wroc{\l}aw University of Science and Technology, Wybrze\.{z}e Wys\-pia\'n\-skiego 27, 50-370 Wroc{\l}aw, Poland.}

\author{Krzysztof Omiljanowski}
\email{Krzysztof.Omiljanowski@math.uni.wroc.pl}
\address{Mathematical Institute, University of Wroc{\l}aw, Pl. Grunwaldzki 2, 50-384 Wroc{\l}aw, Poland}
\date{\today}

\begin{abstract}
 We consider the Vietoris  hyperspaces $\mathcal S(\R^n)$ of simple closed curves in $\R^n$, $n=2,3$, and their subspaces $\mathcal S_P(\R^2)$ of planar simple closed polygons, $\mathcal K_P$ of polygonal knots, and $\mathcal K_T$ of tame knots. We prove that all the hyperspaces are strongly locally contractible, arcwise connected, infinite-dimensional Cantor manifolds, and $\mathcal S(\R^2)$ and $\mathcal K_T$ are strongly infinite-dimensional Cantor manifolds. Moreover, $\mathcal S_P(\R^2)$ and $\mathcal K_P$ are $\sigma$-compact, strongly countable-dimensional absolute neighborhood retracts.

\end{abstract}
\maketitle

\section{Introduction}

S. B. Nadler, Jr.~\cite[Chapter XIX]{Na}   asked a series of questions about  properties and possible characterizations of the Vietoris hyperspaces $\mathcal A(\R^n)$  of arcs, $\Pa(\R^n)$ of pseudoarcs and    $\mathcal S(\R^n)$ of simple closed curves in the Euclidean space $\R^n$, $n>1$. R. Cauty   showed in~\cite{Ca1}  that $\mathcal A(\R^2)$ is an absorbing set in the Borel class $F_{\sigma\delta}$, so it is  homeomorphic to the  infinite-dimensional space  $c_0$ of real sequences converging to 0 (with coordinate-wise convergence); he also proved  in~\cite{Ca2} that $\Pa(\R^2)$  is homeomorphic to the Hilbert space $l_2$.

Little is known on $\mathcal S(\R^n)$. It was observed in~\cite{Na}  that, for each $n>1$,  $\mathcal S(\R^n)$ is of the first category in itself,  contains  a Hilbert cube, and $\mathcal S(\R^2)$ is homogeneous. Moreover,  $\mathcal S(\R^n)$ is an absolute $F_{\sigma\delta}$-set which is not $G_{\delta\sigma}$, since the true $F_{\sigma\delta}$-space  $c_0$ embeds in it as a closed subset~\cite{K1},~\cite{K1e}.
Establishing that a space is an absolute neighborhood retract and finding its exact Borel  class are  first steps toward a topological characterization by  methods of infinite-dimensional topology.

In this paper we focus on Vietoris hyperspaces $\mathcal S(\R^n)$ for $n\in\{2,3\}$ and some of their special subspaces. For $n=2$ we will consider the hyperspace of simple closed polygons $\mathcal S_P(\R^2)\subset \mathcal S(\R^2)$.  Simple closed curves in $\R^3$ are called knots, therefore we use special notation in that case:

$\mathcal K:=\mathcal S(\R^3)=$
the hyperspace of all knots in $\R^3$,

$\mathcal K_P:= \text{the hyperspace of polygonal knots}$,

$\mathcal K_T:= \text{the hyperspace of tame knots}$,

 $\mathcal K_W:= \text{the hyperspace of wild knots}$.

We assume that a knot $C$ is \emph{tame} if there is an autohomeomorphism of $\R^3$ transferring $C$ onto a polygonal knot; otherwise, $C$ is \emph{wild}.

In Section 2 we prove a stronger form of  local contractibility of hyperspaces   $\mathcal S_P(\R^2)$, $\mathcal S(\R^2)$,  $\mathcal K_P$, and  $\mathcal K_T$. We use it to show the arcwise connectedness of $\K$, $\mathcal K_P$, and  $\mathcal K_T$.

  Let us recall two related results concerning the complexity of spaces of knots:
 \begin{itemize}
\item
The set $\{f: S^1\to \R^3: \text{$f$ is an embedding such that $f( S^1)$ is wild}\}$ is comeager  in the space of all embeddings  with the uniform convergence topology~\cite{Mi}.
\item
The relation of knot equivalence on $\mathcal K$ and on $\mathcal K_W$  is not classifiable by countable structures~\cite{Ku}.
\end{itemize}
The issue  was also discussed, e.g.,  in~\cite{MO1, MO2}.

We are interested in the classical Borel complexity of   hyperspaces $\mathcal S_P(\R^2)$, $\mathcal K_P$, $\mathcal K_T$, and $\mathcal K_W$.
 In Section 3,  the exact Borel class of $\mathcal S_P(\R^2)$ and $\mathcal K_P$ is evaluated as   $F_\sigma$ and not $G_\delta$-subsets of the Vietoris hyperspaces $C(\R^2)$ and  $C(\R^3)$ of all continua in $\R^2$ and $\R^3$, respectively.   We also reestablish the fact  that $\mathcal K_T$ and $\mathcal K_W$ are Borel subsets of  $C(\R^3)$. Their exact Borel classes are unknown but  $\mathcal K_T$ is not $G_{\delta\sigma}$.

 In Section 4, hyperspaces $\mathcal S(\R^2)$ and $\mathcal K_T$ are shown to be strongly infinite-dimensional Cantor manifolds each of whose  nonempty open subsets contains a  closed copy of the space $$\hat{c_0}=\{(x_j)_{j=1}^\infty\in [0,1]^\N: \lim_j x_j=0\}$$ (which is known to be homeomorphic to $c_0$~\cite{DMM}); hyperspaces
$\mathcal S_P(\R^2)$ and $\mathcal K_P$ are infinite-dimensional Cantor manifolds containing  closed copies of $$\sigma=\{(x_j)_{j=1}^\infty\in [0,1]^\N: \text{$x_j=0$ for almost all $j$}\}$$ in each nonempty open subset. Finally, it is  observed that  $\mathcal S_P(\R^2)$ and $\mathcal K_P$ are $\sigma$-compact, strongly countable-dimensional absolute neighborhood retracts (ANR's).

\

\section{Strong local contractibility}

Let us call a space $X$  \emph{strongly locally contractible} at $x\in X$ if $X$ has a neighborhood basis at $x$ of open contractible subsets.

 Recall that basic open sets in the Vietoris topology in the hyperspace $C(X)$ of all nonempty continua in $X$ are of the form
$$\langle U_1,\dots,U_n\rangle= \{C\in C(X):  C\subset \bigcup_{k=1}^n U_k\ \text{and}\  \forall\,k\  (U_k\cap C\neq\emptyset)\,\}, n\in \N,$$ where sets $U_k$ belong to any open base in $X$. If $H:X\to Y$ is a continuous map, then by an \emph{induced map} we mean the continuous map $\mathcal{H}:C(X)\to C(Y)$ defined by  $\mathcal{H}(C)=H(C)$.

Let us fix a natural number $m$, $m\geq 6$,
and the simple closed polygon $P$ in the complex plane $\C$
 with vertices $e^{i\frac{2\pi j}{m}}$,  $j=1,\dots,m$.

For any integer $n>6$,  fix the annulus neighborhoods of $P$:
\\
$U=\left\{tz: z\in P, \; 1-\frac{6}{n} < t <  1+\frac{6}{n}\right\} $,
 $K=\left\{t z: z\in P, \; 1-\frac{1}{n} < t <  1+\frac{1}{n}\right\} $,
\\
with components of boundaries:
\\
$P_+=\left\{tz: z\in P, \; t = 1+\frac{6}{n}\right\} $,
 $\ \ K_+=\left\{tz: z\in P, \; t =  1+\frac{1}{n}\right\} $,
\\
$\ \ P_-=\left\{tz: z\in P, \; t = 1-\frac{6}{n}\right\} $,
$\ \ K_-=\left\{tz: z\in P, \; t =  1-\frac{1}{n}\right\} $.
\\
For $k=1,\ldots,n$, cover $U$ and $K$ with small open sectors:
\\
$\ \ U_k=\left\{z\in U: \frac{2\pi (k-1)}{n}-\frac{1}{n} < \arg z <  \frac{2\pi k}{n}+\frac{1}{n}\right\} $,
\\
 $\ \ K_k=\left\{z\in K: \frac{2\pi (k-1)}{n}-\frac{1}{n} < \arg z <  \frac{2\pi k}{n}+\frac{1}{n}\right\} $
\\ 
  (see Figure 1).

Note that the open polyhedra  $U_1,\dots,U_n$ form a circular chain. 
The collection $\{\langle U_1,\dots,U_n\rangle\}_{n=7,8,\ldots}$ is a local base of $C(\C)$ at $P$.

Denote also:
$$\widetilde{U_k}= U_k\times (-\frac{1}{n},\frac{1}{n}),\quad \widetilde{K_k}= K_k\times (-\frac{1}{n},\frac{1}{n})\quad\text{for each $k\in\{1,\dots,n\}$},
$$
$$\widetilde{U}=U\times (-\frac{1}{n},\frac{1}{n}),\quad \widetilde{K}=K\times (-\frac{1}{n},\frac{1}{n}).$$
Similarly, observe that the collection $\{\langle \widetilde{U_1},\dots,\widetilde{U_n}\rangle\}_{n=7,8,\ldots}$ is a local base of $C(\R^3)$ at $P\times \{0\}$.

A family $\mathcal C$ of subsets of a Euclidean  space is \emph{PL-topological} if, for every $C\in \mathcal C$,  each  image of $C$ under a PL-homeomorphism into the space belongs to $\mathcal C$.
\begin{theorem}\label{t1}
If $\mathcal C$ is any PL-topological family of  continua in $\C$ (in $\R^3$) containing  $P$ ($P\times \{0\}$), then the neighborhood $\langle U_1,\dots,U_n\rangle\cap \mathcal C$ of $P$ in $\mathcal C$ (the neighborhood $\langle \widetilde{U_1},\dots,\widetilde{U_n}\rangle\cap \mathcal C$ of $P\times \{0\}$ in $\mathcal C$) is contractible in itself to $P$ (to $P\times \{0\}$, respectively).
\end{theorem}
\begin{proof}
Firstly, we consider the planar case. The required contraction is defined in three moves.

\noindent
(1) Define  $H_1:U\times [0,1]\to  U$ as follows.
Rays $\arg z=\frac{2\pi j}{m}$, $j=1,\ldots,m$ and $\arg z=\frac{2\pi k}{n}\pm \frac{1}{n}$, $k=1,\ldots,n$ divide
$U$ into trapezoids. Let $T$ be any of them with vertices $a_+,\;b_+,\;b_-,\;a_-$, where $a_+,\;b_+\in P_+$ and $\arg a_+=\arg a_-$.
Segments joining the interior point $m=\frac{1}{4}(a_++a_-+b_++b_-)$ with vertices of $T$ divide $T$ into 4 triangles.
For $t\in[0,1]$,
let $H_1(m,t)=m$ and
$$\begin{array}{rcl}
 H_1(a_\pm,t) & = & \frac{1}{2}(a_+ +a_-)\pm (1-\frac{5}{6}t)\frac{1}{2}(a_+-a_-), \\
 H_1(b_\pm,t) & = & \frac{1}{2}(b_+ +b_-)\pm (1-\frac{5}{6}t)\frac{1}{2}(b_+-b_-), \\
\end{array}$$
and  extend $H_1(\cdot,t)$ linearly over all these triangles.

Observe that $H_1$ is a PL-isotopy and $H_1(U_k\times\{1\})=K_k$,  $k=1,\ldots,n$.

Notice that, for any $t\in [0,1]$,
 $$\text{$H_1(C,t)\in \langle U_1,\dots,U_n\rangle\cap \mathcal C$ for each $C\in \langle U_1,\dots,U_n\rangle\cap \mathcal C$}$$
  which means that the induced map $\mathcal H_1$  on $\langle U_1,\dots,U_n\rangle\cap \mathcal C \times [0,1]$ is a continuous deformation whose image remains in  $\langle U_1,\dots,U_n\rangle\cap \mathcal S_P(\R^2)$.

\

\noindent
(2) Let us fix disjoint $U_i,U_j$, ($1<|i-j|<n-1$). Observe that
\begin{multline}\label{p1}
\text{each $C\in  \langle K_1,\dots,K_n\rangle\cap \mathcal C$ contains a subcontinuum  $I_C$ satisfying either}\\
\text{$I_C\subset \cl(U_i)\cap K$ and $I_C$ intersects both components of $\bd(U_i)\cap K$}\\
 \text{or}\\
\text{$I_C\subset \cl(U_j)\cap K$ and $I_C$ intersects both components of $\bd(U_j)\cap K$.}
\end{multline}

\begin{figure}[h]
\begin{center}
\includegraphics[height=14cc]{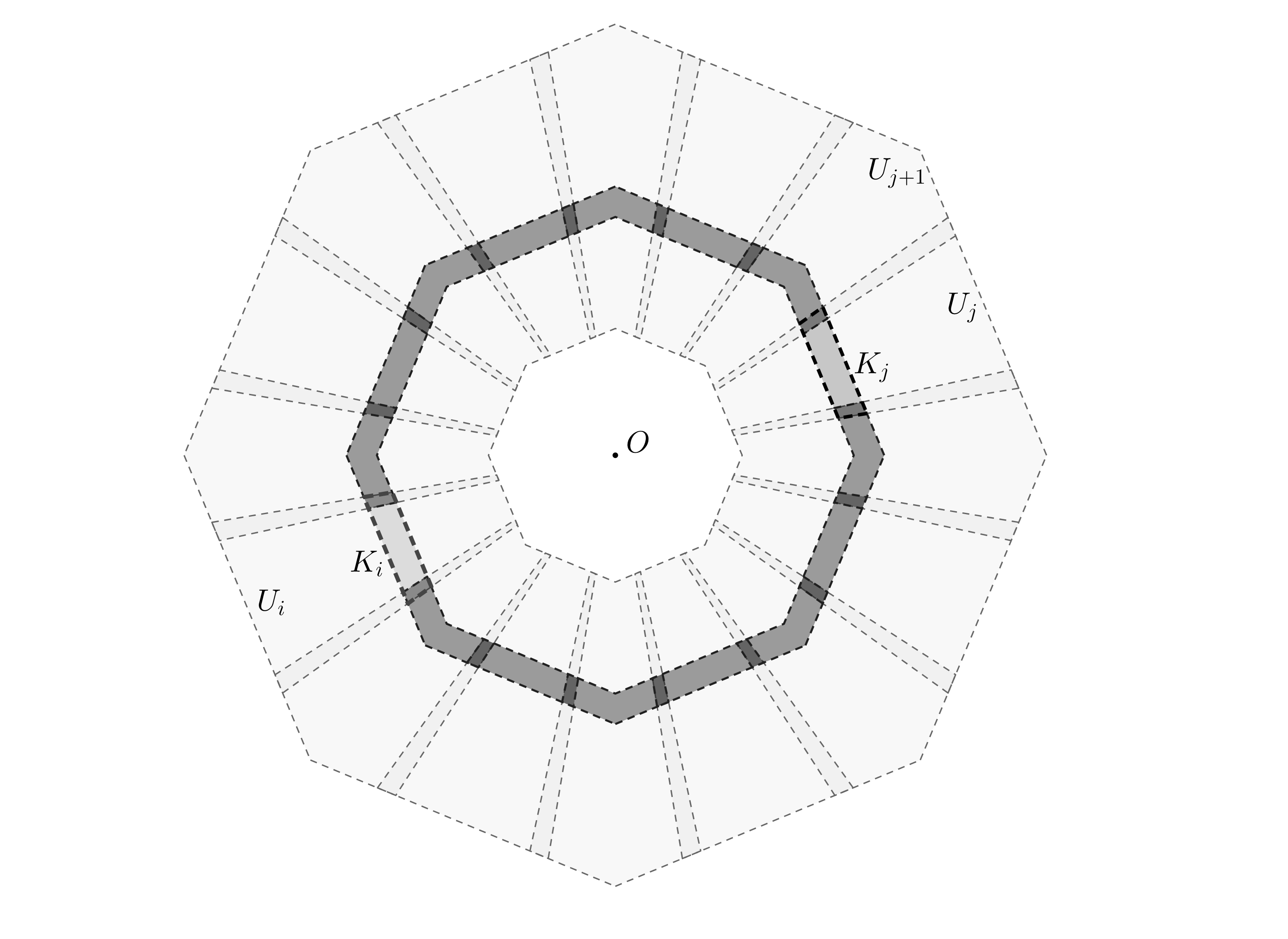}
\end{center}
\caption{$n=2m$}
\end{figure}

Now,  continuously with respect to parameter $t\in[1,2]$ and keeping  points of $K\setminus (U_i\cup U_j)$ fixed,  we PL-deform discs $K_i$ and $K_j$ in $U$ by  folding, stretching,  and wrapping them around $P$,  keeping them ``$\frac{2}{n}$-thin'' and mutually disjoint. The final folded disks for  $t=2$ are  open  disks $K_i'$ and $K_j'$  as pictured in Figure 2.

\begin{figure}[h]
\begin{center}
\includegraphics[height=14cc]{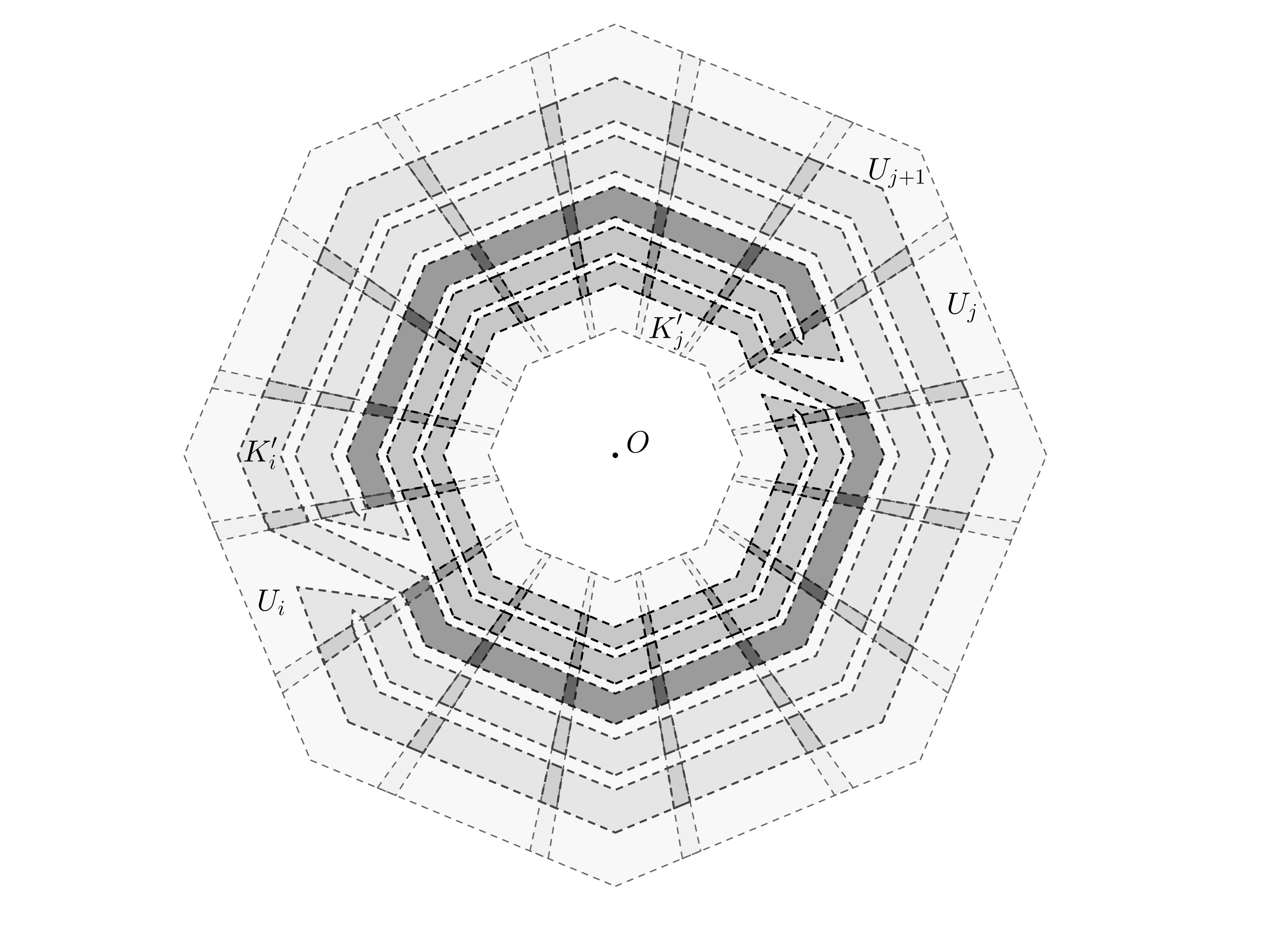}
\end{center}
\caption{$H_2$ deforms $K_i$ to $K_i'$ and $K_j$ to $K_j'$.}
\end{figure}

More formally, the continuous deformation of $K_i$ and $K_j$  can be  realized by a PL-isotopy
$H_2:K\times [1,2]\to U$ such that
\begin{itemize}
\item
$H_2(z,1)=z$ for each $z\in K$,
\item
$H_2(z,t)=z$ for each $z\in K\setminus (U_i\cup U_j)$ and $t\in[1,2]$,
\item
$H_2(K_i,2)= K_i'$ and $H_2(K_j,2)= K_j'$.
\end{itemize}

 The isotopy $H_2$   induces the deformation
 \begin{equation*}
 \mathcal H_2:  \langle K_1,\dots,K_n\rangle\cap \mathcal C\times [1,2] \to  \langle U_1,\dots,U_n\rangle\cap \mathcal C,
    \end{equation*}
   and    it follows from property~\eqref{p1} that, for each
   $C\in  \langle K_1,\dots,K_n\rangle\cap \mathcal C$, the continuum $\mathcal H_2(C,2)$ radially projects onto $P$  (which need not be the case for  $C$ itself).

 \

\noindent
(3)  A pseudo-isotopy $H_3:U\times [2,3] \to U$ is defined similarly to $H_1$.
Namely, for each trapezoid $T$ and its triangulation  define $H_3$ on their vertices:
$H_3(m,t)=m$ and
$$\begin{array}{rcl}
 H_3(a_\pm,t) & = & \frac{1}{2}(a_+ +a_-)\pm (3-t)\frac{1}{2}(a_+-a_-), \\
 H_3(b_\pm,t) & = & \frac{1}{2}(b_+ +b_-)\pm (3-t)\frac{1}{2}(b_+-b_-), \\
\end{array}$$
 for $t\in[2,3]$, and  extend $H_3(\cdot,t)$ linearly over all  triangles of the triangulation.

Note that  $H_3(z,2)=z$ for each $z\in U$, $H_3(z,t)=z$ for $(z,t)\in P\times [2,3]$,
and $H_3(\cdot,t)$  is a PL-homeomorphism for $t\in[2,3)$.
The induced homotopy $\mathcal H_3$ at $t=3$ maps certain continua in $\langle U_1,\dots,U_n\rangle \cap\mathcal C$ onto arcs in $P$; nevertheless, it is well defined on the set   $\mathcal H_2( \langle K_1,\dots,K_n\rangle\cap \mathcal C,2)$ as a map to $\langle U_1,\dots,U_n\rangle\cap \mathcal C$, since continua from the set radially project onto $P$ at $t=3$.

\

Finally,  define a PL-pseudo-isotopy  $H:U\times[0,3]\to U$ :
\begin{equation}
H(z,t)=
\left\{
  \begin{array}{ll}
   H_1(z,t) & \hbox{\text{for $t\in[0,1]$};} \\
   H_2( H_1(z,1),t) & \hbox{\text{for $t\in[1,2]$};} \\
   H_3( H_2( H_1(z,1),2),t) & \hbox{\text{for $t\in[2,3]$}.}
  \end{array}
\right.
\end{equation}
The induced homotopy $\mathcal H$ contracts  $\langle U_1,\dots,U_n\rangle\cap \mathcal C$  to $P$.

\

In the 3-space case of Theorem~\ref{t1},   let us define $\widetilde{H}:\widetilde{U}\times[0,3]\to\widetilde{U}$ by
$\widetilde{H}(x,y,t)=\left(H(x,t),(1-\frac{1}{3}t)y\right)$. Then the induced homotopy $\widetilde{\mathcal H}$ is the required contraction of
$\langle \widetilde{U_1},\dots,\widetilde{U_n}\rangle\cap \mathcal C$ to $P\times \{0\}$.

\end{proof}

Recall that  a space $X$ is said to be \emph{homogeneous} if for each pair of points $x,y\in X$ there is a homeomorphism $f:X\to X$ such that $f(x)=y$; $X$ is \emph{locally homogeneous} if for each pair of points $x,y\in X$ there exist open neighborhoods $U$ of $x$ and $V$ of $y$, and a homeomorphism $f:U\to V$ such that $f(x)=y$.

\begin{proposition}\label{pr1}
The hyperspaces $\mathcal S(\R^2)$ and $\mathcal S_P(\R^2)$ are homogeneous and arcwise connected.
\end{proposition}

\begin{proof}
By the classical result of the plane topology, every two planar simple closed curves (simple closed polygons)  $C_1, C_2$ are isotopic by an ambient isotopy (PL-isotopy) of the plane (see~\cite{Ke}). The induced isotopy defines  an arc between  $C_1, C_2$ in the respective hyperspace $\mathcal S(\R^2)$ and $\mathcal S_P(\R^2)$ as well as a homeomorphism of the hyperspace that maps $C_1$ onto $C_2$.
\end{proof}

\begin{proposition}\label{pr2}
The hyperspaces $\mathcal K_T $ and $\mathcal K_P $ are locally homogeneous.
\end{proposition}

\begin{proof}
Recall that
\begin{multline}\label{k1}
\text{a knot $C$ is tame if and only if there exist  a neighborhood $W$ of $C$}\\ \text{and a homeomorphism  $f:W\to \widetilde{U}$ such that $f(C)=P\times \{0\}$}.
\end{multline}

A counterpart of~\eqref{k1} in the PL-category runs as follows.
\begin{multline}\tag{\ref{k1}$_{PL}$}\label{k2}
\text{$C\in \mathcal K_P$ if and only if there exist  a polyhedral neighborhood}\\ \text{$W$ of $C$  and a PL-homeomorphism  $f:W\to \widetilde{U}$ such that $f(C)=P\times \{0\}$}.
\end{multline}

Let $\overline{f}$ be the induced homeomorphism defined on the hyperspace $C(W)$. Then the preimages $ \overline{f}^{-1}(\langle \widetilde{U_1},\dots,\widetilde{U_n}\rangle\cap \mathcal K_T)$ and $ \overline{f}^{-1}(\langle \widetilde{U_1},\dots,\widetilde{U_n}\rangle\cap \mathcal K_P)$ are open neighborhoods of $C\in \mathcal K_T$ and of $C\in \mathcal K_P$, respectively, such that $\overline{f}(C)= P\times \{0\}$.
\end{proof}

\

\begin{proposition}\label{pr3} Hyperspaces $\K$, $\K_P$, and $\K_T$ are arcwise connected.
\end{proposition}

\begin{proof}
By Theorem~\ref{t1}, it is enough to exhibit, for any knot $C$, a PL-isotopy  on $\R^3$ transferring $C$ into $\langle \widetilde{U_1},\dots,\widetilde{U_n}\rangle$.

There exists an open cube $Q=(-M,M)^3$ containing $C$.
Note that $Q'=\widetilde{U_2}\cup \widetilde{U_3}\cup\ldots\cup\widetilde{U_n}$ is an open polyhedron homeomorphic to $Q$ and there is a PL-isotopic deformation $I:\R^3\times [0,1]\to \R^3$ expanding $Q$ to whole $Q'$
(i.e. $I( Q\times\{1\})=Q'$). Moreover, we can assume that
$$I(C\times\{1\})\cap (\widetilde{U_1}\cap \widetilde{U_2})\not= \emptyset \not=I(C\times\{1\})\cap (\widetilde{U_1}\cap \widetilde{U_n}).$$
By the connectedness, $I(C\times\{1\})$ intersects all the sets $\widetilde{U_k}$, hence
$$I(C\times\{1\})\in \langle \widetilde{U_1},\dots,\widetilde{U_n}\rangle.$$

\end{proof}

Theorem~\ref{t1} applies to families $\mathcal C\in\{\mathcal S_P(\R^2), \mathcal S(\R^2), \K_P, \K_T\}$. Hence, by Propositions~\ref{pr1},~\ref{pr2},~\ref{pr3}, we get our main result in this section.

\begin{theorem}\label{t2}
Each of the hyperspaces $\mathcal S_P(\R^2)$, $\mathcal S(\R^2)$, $\K_P$, and $\K_T$ is strongly locally contractible and arcwise connected.
\end{theorem}

\section{Borel complexity}
\begin{theorem}\label{t4}
Each of the hyperspaces $\mathcal S_P(\R^2)$ and $\mathcal K_P$ is an  $F_\sigma$ and not $G_\delta$-subset of $C(\R^2)$ and  $C(\R^3)$, respectively.
\end{theorem}

\begin{proof}
 We show the theorem for $\mathcal S_P(\R^2)$, the proof being similar for $\mathcal K_P$.
Note that the family of closed segments, including  singletons, in $\R^2$  is a closed subset of $C(\R^2)$. For $A,B \in C(\R^2)$, denote $\rho(A,B)=\min\{||a-b||: a\in A, b\in B\}$.

The following formula describes simple closed polygons  as continua which are  unions of finite circular chains of segments   and  do not contain triods.
\begin{multline}\label{k2}\mathcal S_P(\R^2)= \bigcup_{k\ge 4, l\ge 1} K_{k,l}, \quad\text{where, for $k\ge 4, l\ge 1$,}\\
K_{k,l} =\{C\in C(\R^2): \exists\, x_1,\ldots,x_k\in\R^2 \quad C=[x_k,x_1]\cup\bigcup_{i=1}^{k-1}\,[x_i,x_{i+1}]\quad\&\\
    \forall\, i,j\in \{1,\ldots,k\}\ \text{if $1<|i-j|<k-1$, then $\rho([x_i,x_{i+1}],[x_j,x_{j+1}])\ge \frac1{l}$}\}.
 \end{multline}
The Borel class  of   $K_{k,l}$  is $F_\sigma$ in $C(\R^2)$, whence $\mathcal S_P(\R^2)$ is $F_\sigma$. Since $\mathcal S_P(\R^2)$ contains a closed copy of  $\sigma$ (see Theorem~\ref{t3}) and $\sigma$  is not a $G_\delta$-subset of the Hilbert cube, $\mathcal S_P(\R^2)$ is not $G_\delta$ either.
\end{proof}

\begin{remark}\label{re}
One can easily observe that each   $C\in K_{k,l}$ in~\eqref{k2} can be approximated  by simple closed polygons from other sets $K_{k',l'}$ with $l'>l$. Thus, $\mathcal S_P(\R^2)$ and $\mathcal K_P$ are of the first category in itself.
\end{remark}

\

Let $G$ be the Polish group of autohomeomorphisms of  $\R^3$ with the compact-open topology. We will consider the continuous action of $G$   on $\K_T$ by induced homeomorphisms. The following two facts are known (see~\cite{MO1}), nevertheless we suggest an alternative elementary argument, due to A. Hohti, for the first one without referring to knot diagrams.
\begin{fact}\label{f1}
The  action of $G$ on $\K_T$ has countably many orbits.
\end{fact}

\begin{proof}
 Orbits of tame knots contain polygonal knots, so it suffices to observe that for each $C\in \K_P$ there is a homeomorphism $g:\R^3\to\R^3$ such that $g(C)$ is a polygonal knot with all vertices in $\Q^3$.

\end{proof}

Since, by~\cite{RN},  the orbit of each $C\in \K_P$ is a Borel subset of $C(\R^3)$ and $\K_T$ contains a closed copy of $\hat{c_0}$ (by Theorem~\ref{t3})  which is not a $G_{\delta\sigma}$-subset of the Hilbert cube,  the next fact follows.
\begin{fact}\label{t5}
$\mathcal K_T$ is a Borel subset of $C(\R^3)$ which is not $G_{\delta\sigma}$.
\end{fact}

It was mentioned in the Introduction that $\K$ is an $F_{\sigma\delta}$-subset of $C(\R^3)$.
\begin{corollary}
 $\mathcal K_W$ is   a Borel subset of $C(\R^3)$ which is not $F_{\sigma\delta}$.
\end{corollary}

Denote by $G^+$ ($G^+_{PL}$) the subgroup of $G$ consisting of orientation preserving autohomeomorphisms (PL-autohomeomorphisms).

Recall the following fact.

\begin{fact}\label{fact}
 Knots $C,D$ belong to the same orbit of the group $G^+$ ($G^+_{PL}$) if and only if $C$ and $D$ have the same isotopy type (PL-isotopy type).
\end{fact}

The following theorem can be proved using an idea similar to the proof of Proposition~\ref{pr3}.
\begin{theorem}\label{f2} Each orbit of $G^+_{PL}$ on $\K_T$ is  dense in $\K_T$.
\end{theorem}

\begin{proof}

Choose arbitrary $C\in \K_T$ and $D\in \K_P$. Consider  a basic neighborhood  of $D$ homeomorphic to $\langle \widetilde{U_1},\dots,\widetilde{U_n}\rangle \cap \K_T$ of the form $\langle W_1,\dots,W_n\rangle\cap \K_T$, where $W_k$ is an open polyhedral cube  homeomorphic to $\widetilde{U_k}$ for each $k$, and the union
$Q'=W_2\cup\dots\cup W_n$ is  a polyhedron homeomorphic to an open  cube $Q=(-M, M)^3$ containing $C$. As in the proof of Proposition~\ref{pr3}, replacing sets $\widetilde{U_k}$ with $W_k$, construct a PL-isotopic deformation $I:\R^3\times [0,1]\to \R^3$ such that the homeomorphism $g:=I(\cdot,1):\R^3\to \R^3$  satisfies  $g(C)\in \langle W_1,\dots,W_n\rangle\cap \K_T$. So, we have shown that each polyhedral knot $D$ is arbitrarily close to an element of the orbit of $C$. It follows that any element of the orbit of $D$ is arbitrarily close to an element of the orbit of $C$.

\end{proof}

\begin{remark}\label{rem}
Recall that Fact~\ref{fact} and Theorem~\ref{f2} have their analogs in the plane for the group $G$ of autohomeomorphisms of $\R^2$ and its subgroup  $G_{PL}$ of PL-autohomeomorphisms:  $G$ acts transitively on  $\mathcal S(\R^2)$,  every $C,D\in \mathcal S(\R^2)$ ($C,D\in \mathcal S_P(\R^2)$) have the same isotopy type (PL-isotopy type) in $\R^2$, and orbits of $G_{PL}$ are dense in $\mathcal S(\R^2)$ (see, e.g.,~\cite{Ke}).
\end{remark}

\section{Cantor manifolds and ANR's}
For  definitions and basic properties of various types of infinite dimensions, in particular of the strong countable dimension and weak (strong) infinite dimension, the reader is referred to~\cite{Eng}.
An infinite-dimensional space $X$ is called an \emph{infinite-dimensional  Cantor manifold} if no finite-dimensional closed subset separates $X$. If no weakly infinite-dimensional closed subset separates $X$, then $X$ is \emph{strongly infinite-dimensio\-nal  Cantor manifold}.
The Hilbert cube $[0,1]^\N$ and $\hat{c_0}=\{(x_j)_{j=1}^\infty\in [0,1]^\N: \lim_j x_j=0\}$ are examples of strongly infinite-dimensional  Cantor manifolds. The space $\sigma=\{(x_j)_{j=1}^\infty\in [0,1]^\N: \text{$x_j=0$ for  all but finitely many $j$}\}$ is an infinite-dimensional  Cantor manifold.

\begin{theorem}\label{t3}
\begin{enumerate}
\item The hyperspaces  $\mathcal S(\R^2)$ and $\K_T$ are strongly infinite-dimensional  Cantor manifolds that contain closed copies of $\hat{c_0}$ in each basic set.
\item
 The hyperspaces  $\mathcal S_P(\R^2)$ and $\K_P$ are infinite-dimensional  Cantor manifolds that contain closed copies of $\sigma$ in each basic set.
\end{enumerate}
\end{theorem}

\begin{proof}
First, construct a special  embedding $f: [0,1]^\N \to C(\R^2)$. Fix a sequence $\frac{2\pi}{m}=\alpha_1 >\alpha_2>\dots >0$      converging to $0$ and let $p_j$ be the point of intersection of the ray from $0$ through $e^{i\alpha_j}$ with polygon $P$, for each $j\ge 1$.
Given a sequence $(x_j)_{j=1}^\infty\in [0,1]^\N$, let $f((x_j)_{j=1}^\infty)$ be the closure of
 the union of segments $$[e^{i\frac{2\pi k}{m}},  e^{i\frac{2\pi (k+1)}{m}}],\quad [p_{2j-1}, (1+\frac{x_j}{n})p_{2j}],\ \text{and}\  [(1+\frac{x_j}{n})p_{2j},p_{2j+1}]$$ for $k=1,\dots,m-1$, $j\ge1$. Then $f((x_j)_{j=1}^\infty)\in \langle U_1,\dots,U_n\rangle$ and
 $f$ is an embedding satisfying $$f((x_j)_{j=1}^\infty)\in \mathcal S(\R^2)\quad \text{if and only if}\quad (x_m)_{m=1}^\infty\in \hat{c_0}$$
and $$f((x_j)_{j=1}^\infty)\in \mathcal S_P(\R^2)\quad \text{if and only if}\quad (x_m)_{m=1}^\infty\in \sigma.$$
Also, the mapping $\widetilde{f}((x_j)_{j=1}^\infty)= f((x_j)_{j=1}^\infty)\times \{0\}\in \langle \widetilde{U_1},\dots,\widetilde{U_n}\rangle$ is an embedding satisfying
$$\widetilde{f}((x_j)_{j=1}^\infty)\in \K_T \quad \text{if and only if}\quad (x_m)_{m=1}^\infty\in \hat{c_0}$$ and
$$\widetilde{f}((x_j)_{j=1}^\infty)\in \K_P \quad \text{if and only if}\quad (x_m)_{m=1}^\infty\in \sigma.$$
Equivalently, the above properties mean that $f(\hat{c_0})=f([0,1]^\N )\cap \mathcal S(\R^2)$, $f(\sigma)=f([0,1]^\N )\cap \mathcal S_P(\R^2)$,
$\widetilde{f}(\hat{c_0})=\widetilde{f}([0,1]^\N)\cap \K_T$, $\widetilde{f}(\sigma)=\widetilde{f}([0,1]^\N)\cap \K_P$. Hence, $f(\hat{c_0})$,  $f(\sigma)$, $\widetilde{f}(\hat{c_0})$, and $\widetilde{f}(\sigma)$ are closed subsets of the neighborhoods $\langle U_1,\dots,U_n\rangle \cap \mathcal C$ or $\langle \widetilde{U_1},\dots,\widetilde{U_n}\rangle \cap \mathcal C$ for  respective hyperspaces $\mathcal C\in\{\mathcal S(\R^2), \mathcal S_P(\R^2), \K_T, \K_P\}$.
Since the open base of contractible sets in $\mathcal C$, as in Theorem~\ref{t2}, consists of homeomorphic images of  $\langle U_1,\dots,U_n\rangle \cap \mathcal C$ or $\langle \widetilde{U_1},\dots,\widetilde{U_n}\rangle \cap \mathcal C$  under induced homeomorphisms of the entire plane or the rings $\widetilde{U}$, $n=7,8,\dots$, each  basic set contains a closed copy of $\hat{c_0}$ or $\sigma$, respectively.

Let a closed subset $\mathcal Z$ of $\mathcal C\in\{\mathcal S_P(\R^2), \mathcal S(\R^2), \K_P, \K_T\}$ separate $\mathcal C$, i.e.,  $\mathcal C\setminus \mathcal Z= \mathcal V \cup \mathcal V'$, where  $\mathcal V $ and $\mathcal V'$ are nonempty, open in $\mathcal C$, and disjoint.   Since $\mathcal C$ is connected by Propositions~\ref{pr1} and~\ref{pr3}, $\mathcal Z$ is nonempty. We can also assume $\mathcal Z=\bd_{\mathcal C}\mathcal V=\bd_{\mathcal C}\mathcal V'$, where $\bd_{\mathcal C}$ stands for the ``boundary in $\mathcal C$'' operator.

By Propositions~\ref{pr1} and~\ref{pr2}, there exists a  homeomorphism $\mathfrak h$ which maps  $\langle U_1,\dots,U_n\rangle \cap \mathcal C$ or  $\langle \widetilde{U_1},\dots,\widetilde{U_n}\rangle \cap \mathcal C$, depending on  whether
$$\mathcal C\in\{\mathcal S_P(\R^2), \mathcal S(\R^2)\}\quad\text{or}\quad \mathcal C\in\{\K_P, \K_T\},$$  onto an open neighborhood $\mathcal B$ in $\mathcal C$ of $\mathfrak h(P) \in \mathcal Z$.
Denote by $\mathcal A$  either of the sets:

  $\mathfrak h(f(\hat{c_0}))$  if $\mathcal C=\mathcal S(\R^2)$,

  $\mathfrak h(\widetilde{f}(\hat{c_0})$ if $\mathcal C=\K_T$,

  $\mathfrak h(f(\sigma))$ if $\mathcal C=\mathcal S_P(\R^2)$,

 $\mathfrak h(\widetilde{f}(\sigma))$ if $\mathcal C=\K_P$.

  For $\mathcal C\in\{\mathcal S(\R^2), \K_T\}$ suppose that $\mathcal Z$
 is weakly infinite-dimensional and for $\mathcal C\in\{\mathcal S_P(\R^2), \K_P\}$ suppose that $\mathcal Z$
  is finite-dimensional.

Assume that $\mathcal A \subset \cl_{\mathcal C}\mathcal V$, where $\cl_{\mathcal C}\mathcal V$ is the closure (in $\mathcal C$)    of $\mathcal V$.
Since  $\mathcal A$ is strongly infinite-dimensional  in case (1), and  $\mathcal A$ is infinite-dimensional in case (2), it contains an element $A\in \mathcal V$.
By Fact~\ref{fact}, Theorem~\ref{f2}, and Remark~\ref{rem}, there exist  $B\in \mathcal B \cap \mathcal V'$
 and
 a PL-isotopy $I:\R^2\times [0,1]\to \R^2$ or $I:\R^3\times [0,1]\to \R^3$, respectively, such that $I(A,0)=A$, $I(A,1)=B$. Let $t_0=\max\{t: \mathcal I(\mathcal A, t)\subset \cl_{\mathcal C}\mathcal V\} $, where $\mathcal I$ is the induced homotopy on $\mathcal C$. Notice that $t_0<1$. Since the set
$\mathcal I(\mathcal A, t_0)$ is homeomorphic to $\mathcal A$,  it is not contained in $\mathcal Z$.
 Choose $A_0\in \mathcal I(\mathcal A, t_0)\cap \mathcal V$ and let  $d(A_0, \mathcal C\setminus \mathcal V)= \inf\{d(A_0,C):C\in \mathcal C\setminus \mathcal V\}$, where $d$ is the Hausdorff distance in $\mathcal C$.
Now, by continuity of $\mathcal I$, there is $t>t_0$ such that  $d(\mathcal I(A_0, t), \mathcal C\setminus \mathcal V)>0$. Then $\mathcal I(\mathcal A, t)$ meets both the sets $\mathcal V$ and $\mathcal V'$.  Being a copy of $\mathcal A$,  $\mathcal I(\mathcal A, t)$ is connected, so it intersects $\mathcal Z$. This means that $\mathcal Z$ separates $\mathcal I(\mathcal A, t)$, which  contradicts the fact that $\mathcal A$ is strongly infinite-dimensional Cantor manifold  in case (1), and  $\mathcal A$ is infinite-dimensional Cantor manifold in case (2).

\end{proof}

\begin{theorem}\label{ANR}
Hyperspaces $\mathcal S_P(\R^2)$ and $\K_P$ are $\sigma$-compact, strongly countable-dimensional ANR's. 
\end{theorem}

\begin{proof}
 K. Sakai proved  in~\cite{S} that the Vietoris hyperspace $Pol(X)$ of all connected compact polyhedra in a compact convex subset $X$ of $\R^n$, $n>1$, is $\sigma$-compact, strongly countable-dimensional. Hence, the hyperspace $Pol(\R^n)$ of  connected compact polyhedra in $\R^n$ also is $\sigma$-compact and strongly countable-dimensional. It follows from Theorem~\ref{t4} that $\mathcal S_P(\R^2)$ and $\K_P$ are $\sigma$-compact, and  being such subspaces of respective $Pol(\R^n)$, they are strongly countable-dimensional. Recall that both the hyperspaces are locally contractible (Theorem~\ref{t2}). By a result of W. E. Haver~\cite{H},  each $\sigma$-compact, strongly countable-dimensional and locally contractible metric space is an ANR.

\end{proof}

\section{Questions}

Obviously, hyperspaces $\K$, $\K_W$, or $\K_T$ are not homogeneous by  homeomorphisms induced by  autohomeomorphisms of $\R^3$.

\begin{question}\label{q1}
Are hyperspaces $\K$, $\K_W$, or $\K_T$ homogeneous?
\end{question}

\begin{question}\label{q2}
Are strongly infinite-dimensional hyperspaces $\mathcal S(\R^2)$,  $\K$, $\K_W$, or $\K_T$ ANR's?
\end{question}

\begin{question}\label{q3}
What is the  exact Borel class of $\K_T$  in $C(\R^3)$?
\end{question}

\

\subsection*{Acknowledgments}
We  are grateful to Aarno Hohti for his remarks, especially for suggesting the simple idea of the proof of Fact~\ref{t4}. Our thanks also go to the referee for a careful scrutiny.

\bibliographystyle{amsplain}

\end{document}